\documentclass[12pt]{amsart}
\newtheorem{theorem}{Theorem}[section]
\usepackage[textwidth=1cm]{todonotes}
\usepackage{amsmath}
\usepackage{amssymb}
\usepackage{amsthm}
\usepackage{csquotes}
\usepackage{enumerate}
\usepackage{float}
\usepackage{graphicx}
\usepackage{hyperref}
\usepackage{mathtools}
\usepackage{mdframed}
\usepackage{multirow}
\usepackage{relsize}
\usepackage{scrextend}
\usepackage{tikz}
\usepackage{xparse}
\usepackage{algorithmicx}
\usepackage{algpseudocode}

\newcommand{\seqnum}[1]{\href{https://oeis.org/#1}{\rm \underline{#1}}}
\newcommand{\boardman}[2]{\genfrac{<}{>}{0pt}{}{#1}{#2}}

\algdef{SE}[DOWHILE]{Do}{doWhile}{\algorithmicdo}[1]{\algorithmicwhile\ #1}%

\newtheorem{algorithm}[theorem]{Algorithm}
\newtheorem{definition}[theorem]{Definition}

\newtheorem{corollary}[theorem]{Corollary}
\newtheorem{lemma}[theorem]{Lemma}
\newtheorem{proposition}[theorem]{Proposition}
\newtheorem{remark}[theorem]{Remark}
\theoremstyle{definition}
\theoremstyle{theorem}

\tikzstyle{D0}=[xshift=0cm]
\tikzstyle{D1}=[xshift=2.8cm]
\tikzstyle{D2}=[xshift=5.6cm]
\tikzstyle{D3}=[xshift=8.4cm]
\tikzstyle{D4}=[xshift=11.2cm]
\tikzstyle{U}=[yshift=-1.5cm]
\tikzstyle{place}=[circle,draw=black]
\tikzstyle{post}=[->,shorten >=1pt,>=stealth,semithick]
\tikzstyle{pre}=[<-,shorten <=1pt,>=stealth,semithick]
\usetikzlibrary{3d}
\usetikzlibrary{arrows,snakes,backgrounds}
\usetikzlibrary{calc}
\usetikzlibrary{graphs}
\usetikzlibrary{math} 
\usetikzlibrary{petri}
\usetikzlibrary{positioning}
\usetikzlibrary{shapes.geometric}
\usetikzlibrary{shapes}
\usetikzlibrary{trees}

\tikzset{>=latex}
\tikzset{
  mytree/.style = {
  	scale=1,>=stealth',
  	every node/.style = {draw,align=center},
  	level/.style={sibling distance = 2cm/2^#1, level distance = .5cm}
  	},
  testtree/.style = {
    level 1/.style = {red, level distance = 2.5cm, sibling angle = 90},
    level 2/.style = {blue, sibling angle = 60},
    level 3/.style = {teal},
    grow cyclic
  },
  newtree/.style = {
    level 1/.style = {sibling distance = 16pt, level distance = 7pt},
    level 2/.style = {sibling distance = 8pt},
    level 3/.style = {sibling distance = 4pt},
    level 4/.style = {sibling distance = 2pt},
    level 5/.style = {sibling distance = 1.5pt},
    level 6/.style = {sibling distance = 1pt},
  },
  treenode/.style = {align=center, inner sep=0pt, text centered, font=\sffamily},
  a_leaf/.style = { draw=gray, inner sep=.5pt},
  a_inner/.style = { circle, draw=black, inner sep=.5pt},
  arn_n/.style = {treenode, rectangle, draw=black, text width=2.3em, minimum size=2em},
  arn_rn/.style = {treenode,   circle,  draw=black, text width=1.5em,  thin},
  arn_rx/.style = {treenode, rectangle, draw=black, text width=1.5em, thin, minimum size= 1.5em},
  arn_r2/.style = {treenode, diamond,   draw=black, text width=1.5em, thin, minimum size= 1.5em},
  arn_x/.style = {treenode, rectangle, draw=black, text width=1.3em, minimum size=1.3em}
}

\newlength{\mydiameter}
\NewDocumentCommand{\mynode}{% #1 = node options (optional), #2 = node name, #3 = coordinates, #4 = text, #5 = draw options (optional)
O{}
m
m
m
O{}
}{%
{%
\node [#1] (#2)  at #3 {#4};

\pgfextracty{\mydiameter}{\pgfpointdiff{\pgfpointanchor{#2}{south}}{\pgfpointanchor{#2}{north}}}%

\draw  [#5]  (#2.north west) arc
    [
        start angle=90,
        end angle=270,
        radius={0.5\mydiameter}
    ] -- (#2.south east) arc
    [
        start angle=-90,
        end angle=90,
        radius={0.5\mydiameter}
    ] -- cycle;
}%
}

\newcommand*\circled[1]{\tikz[baseline=(char.base)]{            \node[circle,draw,inner sep=2pt] (char) {#1};}}

\newcommand*\ellipsed[1]{\tikz[baseline=(char.base)]{           \mynode[inner xsep=-2pt]{A}{(0,0)}{#1};}}
       
\begin{document}
\title{Two eggs any style \\ generalizing egg-drop experiments}
\author{Harold R. Parks}
\address {Department of Mathematics\\
 Oregon State University\\ 
 Corvallis\\ 
 Oregon 97331 USA}
\email{hal.parks@oregonstate.edu}
\author{Dean C. Wills}
\address{AppDynamics, San Francisco, California 94107 USA}
\email{dean@lifetime.oregonstate.edu}

\begin{abstract}
	The egg-drop experiment introduced by Konhauser, Velleman, and Wagon, later  generalized by Boardman, is further generalized to two additional types. 
The three separate types of egg-drop experiment under consideration are examined in the context of binary decision trees. It is shown that all three types of egg-drop experiment are binary decision problems that can be solved  efficiently using a non-redundant algorithm---a class of algorithms introduced here. 
The preceding theoretical results are applied to the three types of egg-drop experiment to compute, for each, the maximum height of a building that can be dealt with using a given number of egg-droppings.
\end{abstract}

\maketitle

\section{Introduction}
We generalize the famous egg-drop experiment introduced \mbox{in \cite{konhauser1996way}.}

\begin{quote}
``Suppose we wish to know which windows in a $36$-story building are safe to drop eggs from, 
and which will cause the eggs to break on landing. \ .\,.\,.\ \ Suppose two eggs are 
available. What is the least number of egg-droppings that is guaranteed to 
work in all cases?''
\end{quote}

This experiment was used as an exposition of dynamic programming techniques in \cite{Sniedovich}, examined as a form of weighted binary search in \cite{Wills:2009:CCP:1834851}, and assigned as a exercise in \cite{skiena2009algorithm}; it has also enjoyed wide circulation on the internet
as a possible interview question for programmers as either the egg or light bulb dropping problem \cite{brilliant}.  In the typical
internet version of the problem, the building has been increased 
in size to be 100-stories tall.

Since real eggs seldom survive even being dropped a few feet, some clarification
is called for: We agree that the two eggs are identically strong, each can be reused 
unless it breaks, an egg that breaks from one height will break from any greater height,
and an egg that survives from one height will survive from any lower height.

\medskip
Some years after the problem appeared in \cite{konhauser1996way}, Michael
E. Boardman \cite{BoardEgg} came up with his own solution and, at the 
urging of a colleague,
went on to use his method to solve the problem when the
egg-drop experiment starts with $k$-eggs.  So indeed we will be 
further generalizing beyond Boardman's generalization of the 
original egg-drop experiment.

Boardman's approach revolves around carefully planning ahead to 
consider what  will need to be done when an egg breaks. We might say this is the tactical approach. We want to be more strategic and optimize over the full range of possible outcomes.

Our approach to solving the problem is based on thinking about a 
%\todo[inline]{``binary" inserted}
binary
decision tree that encapsulates a strategy. If the egg breaks, then the tree branches to the left. If the egg survives the fall, then the tree branches to the right. The next node visited either determines where an egg should be dropped, or is a terminal node that reveals the strength of the egg.

Using the binary tree as our guide we are able to handle some interesting variations on the egg-drop experiment with relative ease.
The two variations that lead to the nicest results are the following:
\begin{itemize}
\item
{\bf Replacement Eggs.} The supply of eggs is restored to the original number $k$
whenever the egg that is dropped 
does not break.
\item
{\bf Bonus Eggs.}  A new egg  is received whenever the egg that is dropped 
does not break.
\end{itemize}
%\todo[inline]{``We will refer to the original problem, but starting with $k$ eggs, as Standard Eggs:" replaced}
We will refer to Boardman's $k$-egg version of the original problem as
Standard Eggs:
\begin{itemize}
\item
{\bf Standard Eggs.}  No new eggs will be forthcoming; you break it, you lose it.
\end{itemize}

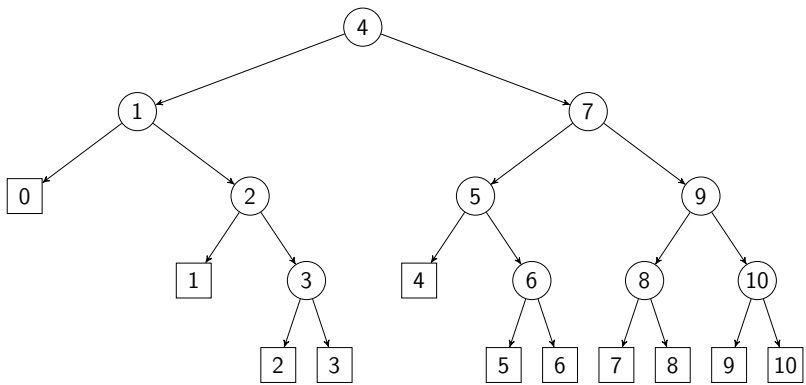
\begin{figure*}[htp]\centering    % TWO EGGS
%\resizebox{\textwidth}{!}{
\resizebox{4.25in}{!}{
\begin{tikzpicture}[->,>=stealth',level/.style={sibling distance = 16cm/2^#1, level distance = 1.5cm}] 

    \node [arn_rn] {4}       % root
        child{ node [arn_rn] {1}    % left
            child { node [arn_rx] {0}   % left left
	               %  child{ node [arn_rx] {0} }   % left left left
		       % child{ node [arn_rx] {1} }   % left left right
    	    }
            child{ node [arn_rn] {2}   % left right 
            	child{ node [arn_rx] {1}   % left right left
%	                  child{ node [arn_rn] {1} } % left right left left
%	                  child{ node [arn_rn] {2} } % left right left right
	         }
		child{ node [arn_rn] {3}  % left right right
	                  child{ node [arn_rx] {2} } % left right right left
	                  child{ node [arn_rx] {3} } % left right right right
	         }
            }
        }
        child{ 
        	node [arn_rn] {7}    % right  
            child { node [arn_rn] {5}   %  right left 
	              child{ node [arn_rx] {4} }   %  right left left 
		      child{ node [arn_rn] {6}   %  right left right 
		                child { node [arn_rx] {5} }  % right left right left
		                child { node [arn_rx] {6} }  % right left right right
		      }
    	    }
            child{ node [arn_rn] {9}  %  right right 
            	  child{ node [arn_rn] {8}   % right right left
	                 child{ node [arn_rx] {7} } % right right left left
	                 child{ node [arn_rx] {8} }  % right right left right
	           }
		  child{ node [arn_rn] {10}   % right right right 
	                 child{ node [arn_rx] {9} } % right right right left
	                 child{ node [arn_rx] {10} }  % right right right right
	           }
            }
        }           
;
\end{tikzpicture}
}

%{\footnotesize Diamond: 2 Eggs,\quad Circle:  1 Egg,\quad Square:  0 Eggs.}
\medskip
\caption{Starting with Two Eggs. No Replacement Eggs. No Bonus Eggs.} \label{eight.two}
\end{figure*}

The tree in Figure~\ref{eight.two} is a binary decision tree solution for two Standard Eggs 
with a $10$-story building. Since $4$ is the number of the root node,  the first egg should be dropped from the 4th floor window. If the egg breaks, we follow the left branch to the node numbered $1$ and that tells us to drop the next egg from the 1st floor window. On the other hand, if the original egg does not break, we follow the right branch from the root to the node numbered $7$ and that tells us to drop the next egg from the $7$th floor window. To summarize:
A node with a number in a circle  
indicates the experiment to perform, i.e., the floor from which to drop the egg; a node with a number in 
a square (always a leaf)  gives the solution, i.e., the strength of the eggs. 

We can see from  the tree  in Figure~\ref{eight.two} 
that four egg drops will suffice for a $10$-story building. We can also see that three egg drops
can only be enough for a $7$-story building, because if we remove the nodes at depth $4$, then there only
$7$ leaves left to represent the strength of the eggs.

%\todo[inline]{I have introduced some foreshadowing in the next paragraph.}
In the next section, we define the notions of a normal binary decision problem (Definition \ref{def.normal})
and a non-redundant algorithm for solving such problems (Definition \ref{def.non.redundant}).  We then
show  in Theorem \ref{the-tree-result}  that any non-redundant algorithm for 
solving a normal problem can be represented by a full binary search tree.
Of course, the point of that work is that it applies to the three egg-drop problems 
described above (and certainly to  many others not yet devised).
In the third section, we use the tree representation of algorithms
to show, for all three variations,  the maximum height of a building that can be dealt with
using a given number of egg-droppings. See Figure~\ref{all.table}. With that information in hand, 
one can answer the fundamental question of egg-drop experiments:
``What is the least number of egg-droppings that is guaranteed to 
work in all cases?'' 

%\todo[inline]{Should we make a small table of number of drops required. The original 36 floors is not enough.}

\begin{figure*}[htp] 
\begin{center}
	
\begin{tabular}{| l | c | c | c | c | c | c | c | c | l |}
\cline{1-1}
{\bf Two Eggs}  \vrule height 2.25 ex width 0 pt depth 0 pt \\[0.25ex]
\hline
\it{Drops} & 1 & 2 & 3 & 4 & 5 & 6 & 7 & 8  & \it{OEIS}\vrule height 2.25 ex width 0 pt depth 0 pt \\[0.25ex]
\hline
Standard       & 1 & 3 & 6 & 10 & 15 & 21 & 28 & 36  & \seqnum{A000217}
\vrule height 2.25 ex width 0 pt depth 0 pt \\[0.25ex]
\hline
Replacement & 1 & 3 & 6 & 11 & 19 & 32 & 53 & 87  &\vrule height 2.25 ex width 0 pt depth 0 pt \\[0.25ex]
\hline
Bonus            & 1 & 3 & 6 & 12 & 22 & 42 & 77 & 147  &\vrule height 2.25 ex width 0 pt depth 0 pt \\[0.25ex]
\hline
\end{tabular}

\bigskip
\begin{tabular}{| l | c | c | c | c | c | c | c | c | l |}
\cline{1-1}
{\bf Three Eggs}  \vrule height 2.25 ex width 0 pt depth 0 pt \\[0.25ex]
\hline
\it{Drops} & 1 & 2 & 3 & 4 & 5 & 6 & 7 & 8  & \it{OEIS}\vrule height 2.25 ex width 0 pt depth 0 pt \\[0.25ex]
\hline
Standard       & 1 & 3 & 7 & 14 & 25 & 41 & 63 & 92   &  \seqnum{A004006} \vrule height 2.25 ex width 0 pt depth 0 pt \\[0.25ex]
\hline
Replacement & 1 & 3 & 7 & 14 & 27 & 51 & 95 & 176 & \vrule height 2.25 ex width 0 pt depth 0 pt \\[0.25ex]
\hline
Bonus            & 1 & 3 & 7 & 14 & 28 & 53 & 103 & 194  &\vrule height 2.25 ex width 0 pt depth 0 pt \\[0.25ex]
\hline
\end{tabular}

\bigskip
\begin{tabular}{| l | c | c | c | c | c | c | c | c | l |}
\cline{1-1}
{\bf Four Eggs}  \vrule height 2.25 ex width 0 pt depth 0 pt \\[0.25ex]
\hline
\it{Drops} & 1 & 2 & 3 & 4 & 5 & 6 & 7 & 8  & \it{OEIS}\vrule height 2.25 ex width 0 pt depth 0 pt \\[0.25ex]
\hline
Standard       & 1 & 3 & 7 & 15 & 30 & 56 & 98 & 162 & \seqnum{A055795}  \vrule height 2.25 ex width 0 pt depth 0 pt \\[0.25ex]
\hline
Replacement & 1 & 3 & 7 & 15 & 30 & 59 & 115 & 223 & \vrule height 2.25 ex width 0 pt depth 0 pt \\[0.25ex]
\hline
Bonus            & 1 & 3 & 7 & 15 & 30 & 60 & 116 & 228  &\vrule height 2.25 ex width 0 pt depth 0 pt \\[0.25ex]
\hline
\end{tabular}
\end{center}
\caption{Maximum height of a building that can be dealt with for small values.}\label{all.table}
\end{figure*}

\section{Theory}

In this section, we consider binary decision problems on integers, i.e., problems where we have the ability to make a binary decision at each stage of execution of an algorithm based on an integral input. We call these decisions {\bf experiments}. Of course, we will later apply this theory to egg-drop problems.
For such problems, it is appropriate to add the following qualifications which simplify and focus 
the discussion.

\begin{definition}\label{def.normal}
	We will call a binary decision problem \textbf{normal} if it meets these criteria:
	\begin{enumerate}
	\item (\textbf{complete}) The solution set is the integers between $0$ and some finite $N$ inclusive. Every solution is possible and constitutes an instance of the problem.
	\item (\textbf{partition}) If an experiment is performed at an integer $x$, success indicates the solution is within the set $[x,N]$, and failure indicates that the solution is within the set $[0,x-1]$. 	
	\end{enumerate}

\end{definition}

\begin{definition} \label{solution}
	We say that $x^{*}$ is the solution of a normal binary decision problem if both of the following hold:
\begin{enumerate}
	\item $x^{*}=0$, or experiment $x^{*}$ is a success, 
\item[] \hskip -2em  and
	\item $x^{*}=N$, or experiment $x^{*}+1$ is a failure.
	\end{enumerate}
\end{definition}

\begin{definition}
We say that an algorithm \textbf{solves a problem} if the algorithm is finite and arrives at a solution for every instance of the problem.
\end{definition}

\begin{definition}\label{def.non.redundant}
	We will call a binary decision algorithm \textbf{non-redundant} if no experiment is performed that is guaranteed to either succeed or fail.
\end{definition}

Deterministically choosing experiments given the results of previous experiments allows us to form a 
binary decision tree\footnote{We follow the convention that the smallest binary tree consists of a single root node.}
 over all possible experiments. In this way, we are performing a search of an ordered table of experiments from \mbox{$1$ to $N$}, trying to determine where the solution is in \mbox{$0$ to $N$}. 

%\todo[inline]{``to indicate an experiment performed at $x$" replaced by ``to indicate that an experiment is to be performed at $x$"}
We can use the notation \circled{$x$} to indicate that an experiment 
is to be performed at $x$, and the notation \fbox{$x$} to indicate that the conclusion of our sequence of experiments is the solution $x$. 
%We can define an order on \circled{$x$}, \fbox{$x$} that conforms to the constraints of Definition \ref{solution} first by number then with \mbox{\circled{$x$} $<$ \fbox{$x$}}. 
%\todo[inline]{Alternative to consider:}
We can define an order on 
$$
\Big\{ \, \circled{$x$}: 1\leq x\leq N\Big\} \cup \Big\{ \,\fbox{$x$} : 0\leq x\leq N\Big\} 
$$
that 
conforms to the constraints of Definition \ref{solution} first by the value of the number then with 
\mbox{\circled{$x$} $<$ \fbox{$x$}}. 

We then find ourselves searching for the solution in an ordered table: 
\begin{center}
  \fbox{$0$}  , \circled{$1$}  , \fbox{$1$}  , \circled{$2$}, \fbox{$2$} , $\cdots $ , \fbox{$N-1$}  , \circled{$N$}, \fbox{$N$}.
\end{center}
This was studied by Knuth in \S 6.2.1 of \cite{knuth1998art3}, where he stated
\begin{displayquote}
``$\dots$\ \textit{any} algorithm for searching an ordered table of length $N$ by means of comparisons can be represented as an $N$-node binary tree in which the nodes are labeled with the numbers \mbox{$1$ to $N$} (unless the algorithm makes redundant comparisons).''
\end{displayquote}

Knuth goes on to present (what we might call) hybrid trees, with internal nodes representing comparisons in circles, and external nodes representing conclusions in boxes, as we have defined 
above.
%\todo[inline]{``\textit{supra}" replaced by ``above"} 
Knuth then asserts, 
\begin{displayquote}
``Conversely, any binary tree corresponds to a valid method for searching an ordered table; we simply label the nodes 

\medskip
%\begin{table}[h]
\centerline{
\begin{tabular}{lllllllll}
  \fbox{$0$}  & \circled{$1$}  & \fbox{$1$}  & \circled{$2$}& \fbox{$2$} & $\cdots $ & \fbox{$N-1$}  & \circled{$N$}& \fbox{$N$}
 \end{tabular}}
%\end{table}

\medskip
\noindent
in symmetric order, from left to right."\footnote{Symmetric order is also called inorder.}
\end{displayquote}

%\todo[inline]{Hal has put in a different wording of the next two paragraphs. The previous wording is preserved as  comments in the TeX source.}
A summary of what Knuth is telling us to do is the following: 
Create an arbitrary binary tree with $N$ nodes,
label the nodes inorder with circled numbers, then add leaf nodes until each of the original nodes has two children, 
finally number these leaf nodes inorder with boxed numbers. 
The result will be a full binary tree that is labeled as described above by Knuth.

%Equivalently, we can create an arbitrary binary tree with $N$ nodes,
%label the nodes inorder with circled numbers, then add leaf nodes until each of the original nodes has two children, 
%numbering these leaf nodes inorder with boxed numbers. 
%This will result in a full binary tree that is labeled with the stated symmetric order (i.e., inorder). 

There is still a bit more work to do to
prove that following the sequence of experiments dictated by the labels on the inner nodes
does in fact lead to the leaf labeled with the solution.

% To show that this suffices, we have a little more work to do. 

\begin{lemma} \label{lemma.successor}
	The inorder successor of an inner node in a full binary tree is the first element of the inorder traversal of the right child subtree. Likewise, the predecessor of an inner node in a full binary tree is the last element of the inorder traversal of the left child subtree.
\end{lemma}

\begin{proof}
	Since the tree is full, an inner node has both left and right child sub-trees, which are non-empty by definition. The result then follows by the definition of inorder.
	\end{proof}

\begin{corollary}\label{corollary.ancestor}
	The inorder successor of a leaf node in a full binary tree is either nil or an ancestor, whose left child subtree contains the leaf. 
%	where the leaf is an element of %the {\bf that ancestor's} left child subtree. 
	Likewise, the inorder predecessor of a leaf node in a full binary tree is either nil or an ancestor, whose right child subtree contains the leaf.
%	where the leaf is an element of 
%the  {\bf that ancestor's}  right child subtree.
\end{corollary}

\begin{proof}
	Follows from Lemma \ref{lemma.successor}.
\end{proof}

\begin{theorem}\label{the-tree-result}
Any non-redundant   algorithm that solves a normal problem over $N$ experiments 
can be  represented by a full binary search tree as above.
\end{theorem}

\begin{proof}
%\todo[inline]{Another bit of rewording.}
Utilization of the binary decision tree will always terminate at a leaf, labeled as a \fbox{$x$} for some $x$. 
It remains to be shown that that $x$ satisfies the conditions of  Definition \ref{solution}. 
%All that remains to be shown is that this constructed tree actually 
%meets the conditions of solution, per Definition \ref{solution}. 

If $x>0$, then it has a inorder predecessor of \circled{$x$}, which by Corollary \ref{corollary.ancestor} is an ancestor whose right child subtree contains \fbox{$x$}. This means that in execution of the binary decision tree, an experiment was performed at $x$ and succeeded. 

If $x<N$, then it has a inorder successor of \ellipsed{$x\mathsmaller{+1}$}, which by Corollary \ref{corollary.ancestor} is an ancestor whose left child subtree contains \fbox{$x$}. This means that in execution of the binary decision tree, an experiment was performed at $x+1$ and failed.
\end{proof}

\subsection{Normality and Non-Redundancy}

We would like to apply the theory developed above to egg-drop problems.
Clearly, we have the next result.

\begin{proposition}\label{prop.normal}
All three egg-drop problem variations, Replacement, Bonus, and Standard, are normal.	
In addition, any variation that loses a constant number of eggs on failure and gains a 
constant number of eggs on success is also normal.
\end{proposition}

%\todo[inline]{``seems" has been italicized}
If we want to  minimize the number of consecutive experiments required, 
then it {\it seems} evident that we should restrict our attention to non-redundant algorithms.
Proving that that restriction can be made requires carefully
analyzing an arbitrary solution algorithm.
For that analysis, we will assume the algorithm is  represented  by a tree,
with the experiments to be preformed in (metaphorical) circles and the solutions
in (metaphorical) squares. Failure of an experiment branches to the left in the tree
and success branches to the right. 
The tree is to represent what the algorithm 
actually does,  so it is 
required that for every node there is at least one  solution that would cause
the algorithm to reach that node: 
Any inaccessible nodes and their descendants
are to be removed.

\begin{definition}\label{solution.range}
Consider a normal problem and a binary tree representing  a solution algorithm
for the problem. 
For every inner node in the tree we associate with it 
the closed interval $[y,z]$ where $y$ and $z$ are the smallest and 
largest labels of leaves that are accessible  descendants of the node.
We will call that interval the \textbf{solution range} of the node. 
\end{definition}

Observe the following facts:
\begin{enumerate}
\item
If the node is the root node, then the solution range is necessarily $[0,N]$, where $N$ is the 
largest experiment. 
\item
If the solution range of a node is $[y,z]$, then $y$ is either $0$ or 
the largest number that is known to succeed
based on the outcomes of previous experiments.
\item
If  the solution range of a node is $[y,z]$, then  $z$ is either $N$ or one less 
than the smallest number known to fail
based on the outcomes of previous experiments.
\end{enumerate}

\begin{proposition}\label{normal-non} 
Any algorithm for solving any of the egg-drop  problem variations  can be modified 
so as to be non-redundant and to never require more experiments than the original algorithm. 
\end{proposition}
\begin{proof}
We will show that any algorithm that requires 
an experiment for which the outcome is guaranteed can be modified
so that that experiment is either omitted altogether or is
replaced by another experiment that does not have a 
guaranteed outcome. In any case, no more  experiments will be required
than were required by the original algorithm. Further modifications can be made 
until all experiments with guaranteed outcomes are eliminated.
This process is explained in detail below:

In the tree representation,
an experiment for which the outcome is guaranteed
will correspond to an inner node with only one child---recall
inaccessible nodes have all been removed. We consider such an inner 
node of least depth, say  $d$. After all nodes at depth $d$ have
been processed, we progress to nodes of greater depth.
When modifications are made to the algorithm, new inaccessible nodes
might be created, so when dealing with a particular node, removing
inaccessible descendants  is the first thing to do.

\medskip
For a node with a guaranteed outcome, let $x$ be the floor from which the
egg is to be dropped
and let $[y,z]$ be the solution range associated with that node. 
Since the situations are different, we will need to
consider separately the case when the guaranteed outcome is failure
and when the guaranteed outcome is success.

If $y=z$, then the solution is already known to be $y$. Since $y$ is known 
to be the solution, the node we are considering should instead be a leaf with $y$ inside 
the square to indicate that $y$ is the solution.
The entire left or  right child subtree below 
the node should be eliminated. The node is now a leaf, and the height of the
tree has not been increased.

Assume now that $y<z$, so  the solution $s$ could be any integer in $[y,z]$, 
and the algorithm cannot yet terminate.

\medskip
\noindent
\textbf{Guaranteed failure.}  Since dropping on $x$ will give no new information
and will lose an egg, this experiment could simply be omitted. Instead,
let us note that for failure to be guaranteed, it must hold that $x>z$.
By dropping the egg on $z$ instead of $x$, we might have a success and that
would tells us that the solution is $z$. If the egg breaks, then we are no
worse off than before. We have gained the information that the egg breaks on $z$,
but we can ignore the additional information for now and follow the
original algorithm.
For the tree, the label of the  node should be changed to $z$,
a right-child leaf labeled $z$ should be added, and the left-child 
subtree is unchanged. The node now has two children, and 
the height of the tree has not been increased. 

\medskip
\noindent
\textbf{Guaranteed success.}
Because success is guaranteed, it must hold that $x\leq y$.
The drop on $x$ is guaranteed to result in the egg surviving, 
so it may be  beneficial  in that  replacement or bonus eggs will be acquired. But since
no new information is gained, 
at least one additional experiment beyond the drop on $x$ will be required.  

Consider what happens if  the drop is made from $y+1$ instead of from $x$.
If the egg breaks, then the solution is known to be $y$. No additional 
experiments are required and the modified algorithm can terminate.
If the egg survives, then the benefit of acquiring replacement or bonus 
eggs is achieved as would have happened with a drop on $x$. 
We have gained the additional information that the egg survives when 
dropped on $y+1$, but we may ignore that information for now and 
simply follow the original algorithm.
For the tree, the label of the node should be changed to $y+1$,
a left-child leaf labeled $y$ should be added, and the right-child
subtree is unchanged. The node now has two children, and 
the height of the tree has not been increased. 

\medskip
The information that the egg broke when dropped on $z$
or  that the egg  survived when dropped on $y+1$ was 
temporarily ignored when we chose to leave the  subtree unchanged. Nonetheless
the information will affect the solution range
of descendant nodes.
We might well find newly inaccessible nodes and   new instances of guaranteed outcome
experiments in the   subtree.  Such issues occur at depth greater than $d$
and are to be dealt with after all depth $d$ nodes 
have been processed. When all nodes at depth $d$ have been processed, they are
all either leaves or inner nodes with two children. 
%\todo[inline]{Added ``then"}
Then when all nodes in the tree have
been processed the tree will be a full binary tree with height no greater than the 
original tree.
\end{proof}

This process of modifying the tree described in the preceding proof 
is also codified in the following algorithm.

\begin{algorithm} \textit{Non-Redundancy}. We assume branch left on failure, $N>0$ and $N+1$ integral solutions in $[0,N]$. For each node, let $\mathcal{N}$ be its solution range $[y,z]$ and label $x$. For nodes other than the root, let $\mathcal{P}$ be the parent of $\mathcal{N}$, so that \textbf{reference}$(\mathcal{P})=\mathcal{N}$. Note that \textbf{reference} resolves to either $\textbf{left}$ or $\textbf{right}$. Let $\mathcal{L}(x)$ be the leaf containing $x$.

\begin{algorithmic}[1]
\For{each Node $\mathcal{N}$, Breadth first,}
\State Unlink all Nodes from $\mathcal{N}$, and its children, that are impossible to visit. \label{alg:inaccessible}
\If{$\mathcal{N}$ is not a leaf}
\If{$y=z$} \label{alg:condition}
	\State unlink $\mathcal{L}(z)$ and set \textbf{reference}$(\mathcal{P})=\mathcal{L}(z)$.\label{alg:change_z0}
\ElsIf{$x>z$} \label{alg:failure}
	\State set \textbf{label}$(\mathcal{N})=z$, unlink $\mathcal{L}(z)$ and set \textbf{right}$(\mathcal{N})=\mathcal{L}(z)$.  \label{alg:change_z}
\ElsIf{$x \leq y$}\label{alg:success}
	\State set \textbf{label}$(\mathcal{N})=y+1$, unlink $\mathcal{L}(y)$ and set \textbf{left}$(\mathcal{N})=\mathcal{L}(y)$. \label{alg:exp_y}
\EndIf
\EndIf
\EndFor
\end{algorithmic}
\end{algorithm}

\begin{proof} 
Line~\ref{alg:inaccessible} implements the removal of inaccessible descendants.  The condition in line~\ref{alg:condition}
is true when the node should be a leaf, and line~\ref{alg:change_z0} converts the node to an appropriately labeled leaf.
The condition in line~\ref{alg:failure} is true when the outcome is guaranteed to be failure, and line~\ref{alg:change_z} changes the 
experiment to $z$ and adds the new leaf that reports the solution is $z$ if the drop on $z$ is a success.
The condition in line~\ref{alg:success} is true when the outcome guaranteed to be success, and line~\ref{alg:exp_y} changes
the experiment to $y+1$ and adds the new leaf that reports the solution is $y$ if the drop
on $y+1$ is a failure. 

When all nodes are processed, and they will be, then each will  either be converted to a leaf 
or will have two children.
\end{proof}

\section{Egg Drop Numbers}\label{section.edn}
Finding solutions for the egg-drop problems with the minimum number of drops turns out to be straightforward. 
Proposition \ref{normal-non} and Theorem \ref{the-tree-result}
tells us that the algorithm we seek can be represented by a full binary search tree. 
Keeping track of eggs in hand, we recursively add nodes to the tree 
breadth first until there are sufficient inner nodes, 
%\todo[inline]{Changed ``or" to  ``representing"}
representing  floors, while not violating the 
constraints of the problem.
For instance, for Standard Eggs starting with $k$ eggs, this 
means no path from the root can branch left more than $k$ times.
This process is carried out in Figure  \ref{BKVW.eggs} starting with $2$ eggs with the
number of eggs remaining  shown at each node.
We then determine the maximum number of inner nodes for depth $d$, call it $H_{{\mathcal P},\,k}(d)$,
where the subscript $\mathcal P$ indicates the problem we are trying to solve, and the subscript $k$
indicates starting with $k$ eggs; we use $\mathcal S$ for Standard eggs.
Then $H_{{\mathcal S},\,k}(d)$ is the highest floor that can be distinguished starting with $k$ eggs; Boardman calls these the ``egg-drop numbers" in \cite{BoardEgg}.
Counting the inner nodes in  Figure  \ref{BKVW.eggs}, we see that $H_{{\mathcal S},\,2}(4)=10.$ 
If it is impossible for eggs to be exhausted after $d$ drops, then the maximum is attained and $H_{*,*}(d)=2^{d-1}$.\footnote{In terms of egg-drop problems, this fact is equivalent to the unsurprising statement that if you have $d$ eggs available, then you can determine their strength for a building up to $2^d-1$ stories tall.}

We note that the type of problem constrains the possible topology of the tree, i.e., since one can not drop
an egg once one has exhausted the supply of eggs, nodes with no eggs are necessarily leaves. 
\begin{enumerate}
	\item The left child always has one fewer egg than the parent.
	\item The right child has:
	\begin{enumerate}
	\item Standard  Eggs: the same number of eggs as the parent, as in Figure~\ref{BKVW.eggs};
	\item Replacement Eggs: the same number of eggs as the root node, as in Figure~\ref{replacement.eggs};
	\item Bonus Eggs: one more egg than the parent, as in Figure~\ref{bonus.eggs}.
	\end{enumerate}
\end{enumerate}

Applying the numbering scheme $\{0, 1, 1, \cdots x, x \cdots 10, 10\}$ to 
Figure~\ref{BKVW.eggs} with an inorder traversal and using circles for inner nodes, squares for leaves,
gives us the annotated binary decision tree Figure~\ref{eight.two}.  We also note that by 
counting the inner nodes
in Figures \ref{replacement.eggs} and \ref{bonus.eggs}, we learn that 
$H_{{\mathcal R},\,2}(4) = 11$ and $H_{{\mathcal B},\,1}(4) = 7$, the subscripts 
$\mathcal R$ and $\mathcal B$ indicating Replacement Eggs and Bonus Eggs, respectively.

%\todo[inline]{We have replaced $S$ by $N$, so those changes have been made.}
\medskip 
To handle a taller building, we will need to grow a bigger tree. Specifically,
to determine how many egg drops are required for a  building
with $N$-stories,
we must construct a tree that has enough depth that the number of
inner nodes is $N$ and the number of  leaves 
is $N+1$.  

This procedure will quickly outgrow the use of paper
and pencil, and will require a computer. Trees are a 
standard data structure and inorder transversal is a standard 
algorithm.  We might say  the problem is solved. But as mathematicians
we would like to do more than simply turn the problem over to a 
computer. In particular,  we would like to answer the
original question of how many egg-drop experiments will be required
or to at least estimate that number.

\subsection{Standard Eggs.} \label{bkvw}
Although complex and clever derivations have been made in the literature for this result, we present a combinatorial argument, which is both concise and elucidating.

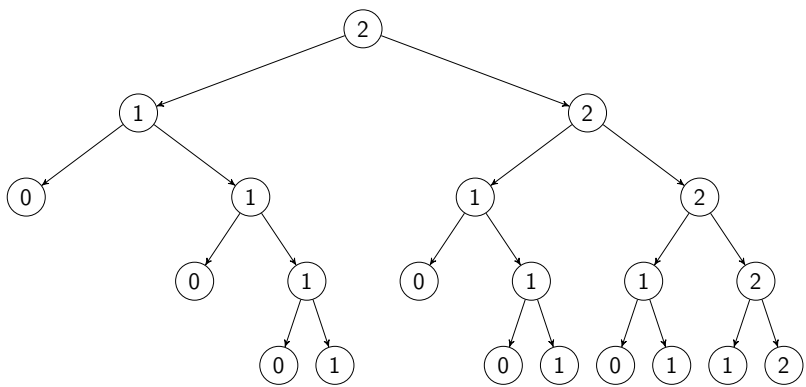
\begin{figure*}[htp]\centering    % TWO EGGS
%\resizebox{\textwidth}{!}{
\resizebox{4.25in}{!}{
\begin{tikzpicture}[->,>=stealth',level/.style={sibling distance = 16cm/2^#1, level distance = 1.5cm}] 

    \node [arn_rn] {2}
child{
   node [arn_rn] {1}
   child{
      node [arn_rn] {0}
   }
   child{
      node [arn_rn] {1}
      child{
         node [arn_rn] {0}
      }
      child{
         node [arn_rn] {1}
         child{
            node [arn_rn] {0}
         }
         child{
            node [arn_rn] {1}
         }
      }
   }
}
child{
   node [arn_rn] {2}
   child{
      node [arn_rn] {1}
      child{
         node [arn_rn] {0}
      }
      child{
         node [arn_rn] {1}
         child{
            node [arn_rn] {0}
         }
         child{
            node [arn_rn] {1}
         }
      }
   }
   child{
      node [arn_rn] {2}
      child{
         node [arn_rn] {1}
         child{
            node [arn_rn] {0}
         }
         child{
            node [arn_rn] {1}
         }
      }
      child{
         node [arn_rn] {2}
         child{
            node [arn_rn] {1}
         }
         child{
            node [arn_rn] {2}
         }
      }
   }
}
          
;
\end{tikzpicture}
}

\caption{Counting the number of eggs in a Standard Egg tree with 2 initial eggs.} \label{BKVW.eggs}
\end{figure*}

\begin{theorem}\label{bkvw-eggs}
	With Standard Eggs, 
	the height of the tallest building for which the strength of the eggs can be determined starting with $k$ eggs and  
	using no more than \mbox{$d$ egg drops}\footnote{This result 
	is consistent with the egg drop number  $\boardman{d}{k}$ defined by Boardman, and a combinatorial argument based on words appears in the section ``Direct counting approach" of \cite{BoardEgg}.} is 
$$H_{{\mathcal S},\,k} (d) = \sum_{i=1}^k  \binom{d}{i}.$$
\end{theorem}

\begin{proof}
%\todo[inline]{The term ``drop tree" appeared here for the first and last time. It has been replaced by ``tree representing the algorithm"}
The number of leaves in the 
% drop tree 
tree representing the algorithm
is identical to the number of paths from the root to the leaves with $i$ breaks,
$i$ running from $0$ up to $k$.  Each such path uniquely corresponds to a binary word
with a $0$ representing going left and a $1$ representing going right. If the height of the
tree is $d$, we would like to consider binary words of length $d$.  But if there are $k$ zeros in the word,
that means all the eggs were broken, the path ends, and so does the word. 
Nonetheless, if we put additional $1$'s after the last of the $k$ zeros, we can bring the length
of the word up to $d$ while maintain the one-to-one correspondence between binary words and paths
in the tree.

For a particular $i$,  the number of binary words of length $d$ that
contain exactly $i\ \ 0$'s is  $\binom{d}{i}$. 
Adding this over all possible $i$ from $0$ to $k$, we get
\[
\sum_{i=0}^k  \binom{d}{i}
\,.
\]
The number of inner nodes is $1$ fewer than the number of leaves, so
\[
H_{{\mathcal S},\,k}(d) =  \sum_{i=1}^k  \binom{d}{i}
\,.
\]
\end{proof}

\subsection{Replacement Eggs.}
%\todo[inline]{added some words about digressing}
To determine the egg-drop numbers for Replacement Eggs, we need to make 
a small digression into the topic of $k$-bonacci numbers.

Recall that the $k$-bonacci numbers $F_\ell^{(k)}$, ($k \geq 1$) are a generalization of
the Fibonacci numbers  defined by the initial values
\begin{equation}\label{fib.init}
F_\ell^{(k)} = 
\begin{cases}
       0, & 0 \le \ell<k-1, \\
       1, & \ell=k-1, \\
\end{cases}
\end{equation}
and the recursion
\begin{equation}\label{fib.recur}
F_\ell^{(k)} =  \sum_{i=1}^k F_{\ell-i}^{(k)},\hbox{\rm\ \ for\ \ }k\leq \ell
\,.
\end{equation}

Following the idea used in the proof of Theorem \ref{bkvw-eggs}, 
we see that we should examine the number of binary words of a
given length $n$ that \textbf{do not} contain $k$ consecutive $0$'s.

\begin{lemma} \label{consecutive}
There are $F^{(k)}_{n+k}$ binary words of length $n \geq 0$ lacking $k$ consecutive $0$'s.
\end{lemma}

%\todo[inline]{The actual proof is simpler than following up on the references in  the remark, so I put the proof back in.}

\begin{proof}
Let $n$ be the length of a binary word. For $n<k,$ all words qualify, and recall that $F^{(k)}_{n+k} = 2^k$ for $0 \leq n<k$. For $n \geq k$, every qualifying word must have a trailing $1$ followed by $i<k$ zeros. This word can be formed by a qualifying word of length $n-i-1$ followed by the $1$ and $i$ zeros. The total number of ways to do this is
\[
\sum_{i=0}^{k-1} F^{(k)}_{n-i-1+k} = \sum_{i=1}^{k} F^{(k)}_{n+k-i} = F^{(k)}_{n+k}
\,.
\]
\end{proof}

\begin{remark}
Lemma \ref{consecutive}  is also implied by the definition of $p_n(k)$, and equations $(1)$, $(10)$, and $(12)$ in \cite{Shane:1973:FPF}, using the version of the $k$-bonacci numbers (there $n$-bonacci) shifted so that $F_{n,0}=1$ and $F_{n,-r}=0$. 
This was built upon and clarified by \mbox{Lemma 2.2} in \cite{Philippou:1982:WTC}, as $A_n^{(k)}=f^{(k)}_{n+1}$, this time using $f^{(k)}$ for the $k$-bonacci numbers with the same shift, which translates to the result with a small amount of algebraic manipulation.
\end{remark}

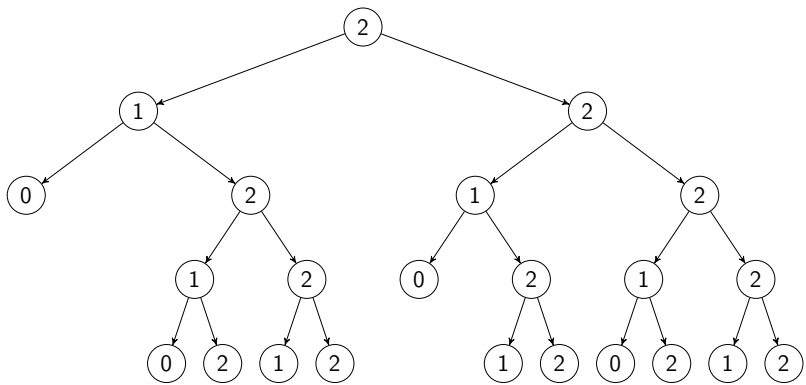
\begin{figure*}[htp]\centering    % TWO EGGS
%\resizebox{\textwidth}{!}{
\resizebox{4.25in}{!}{
\begin{tikzpicture}[->,>=stealth',level/.style={sibling distance = 16cm/2^#1, level distance = 1.5cm}] 

    \node [arn_rn] {2}
child{
   node [arn_rn] {1}
   child{
      node [arn_rn] {0}
   }
   child{
      node [arn_rn] {2}
      child{
         node [arn_rn] {1}
         child{
            node [arn_rn] {0}
         }
         child{
            node [arn_rn] {2}
         }
      }
      child{
         node [arn_rn] {2}
         child{
            node [arn_rn] {1}
         }
         child{
            node [arn_rn] {2}
         }
      }
   }
}
child{
   node [arn_rn] {2}
   child{
      node [arn_rn] {1}
      child{
         node [arn_rn] {0}
      }
      child{
         node [arn_rn] {2}
         child{
            node [arn_rn] {1}
         }
         child{
            node [arn_rn] {2}
         }
      }
   }
   child{
      node [arn_rn] {2}
      child{
         node [arn_rn] {1}
         child{
            node [arn_rn] {0}
         }
         child{
            node [arn_rn] {2}
         }
      }
      child{
         node [arn_rn] {2}
         child{
            node [arn_rn] {1}
         }
         child{
            node [arn_rn] {2}
         }
      }
   }
}

;
\end{tikzpicture}
}
\caption{Replacement Eggs.} \label{replacement.eggs}
\end{figure*}

\begin{theorem}
With Replacement Eggs, the height of the tallest building for which the strength of the 
eggs can be determined starting with $k$ eggs and  
using no more than \mbox{$d$ egg drops} is 
$$
H_{{\mathcal R},\,k}(d) = -1 + \sum_{i=0}^{\lfloor d/(k+1) \rfloor} (-1)^{i} \binom{d-ik}{i}2^{d-i(k+1)}
\,.
$$
\end{theorem}

\begin{proof}
For replacement eggs, we can only exhaust our eggs if we get $k$ consecutive breaks; 
any other sequence will lead to eggs being replenished. We can look at every path from the 
root as a binary word, using $0$ for a break and $1$ for a success. 

We will count the number of binary words corresponding to a tree of height $d$. 
For any of our words that are shorter than $d$ bits, we would like to add
$0$'s to bring them up to length $d$. Such words shorter than
$d$ must end in $k$ consecutive $0$'s, so if they are distinct
before adding $0$'s, they remain distinct after $0$'s are added.

One word consists only of $0$'s.   Every other word
will contain at least one $1$. Since there is a $1$ in the word, 
there is a last $1$ in the word and that $1$,
if followed by anything, is followed by $0$'s

To count the number of words that contain a $1$, remove the last $1$ in the
word and remove all the $0$'s that follow that last $1$.
What remains is 
a binary word lacking $k$ consecutive $0$'s having length somewhere between $0$ and $d-1$.
Thus, the number of leaves in the tree, $L(d)$, is  $1$ for the single path without any $1$'s plus
the sum over $i$ going from $0$ to $d-1$ of the number of  binary words 
lacking $k$ consecutive $0$'s having length $i$. By Lemma \ref{consecutive}, the number
of leaves is
$$
L(d) = 1+ \sum_{i=0}^{d-1}  F^{(k)}_{i + k} 
= 1+ \sum_{i=k}^{d+k-1}  F^{(k)}_{i}
=   \sum_{i=0}^{d+k-1}  F^{(k)}_{i}
\,,
$$
where we have used the $k$-bonacci initial conditions (\ref{fib.init}).

In fact, the partial sums of the sequence  $k$-bonacci numbers can be expressed in terms of
a sum of products of binomial coefficients and powers of $2$ as follows:
\begin{equation}\label{dunkel}
\sum_{i=0}^{d+k-1}  F^{(k)}_{i} =  \sum_{i=0}^{\lfloor d/(k+1) \rfloor} (-1)^{i} \binom{d-ik}{i}2^{d-i(k+1)}
\,.
\end{equation}
%\todo[inline]{Reworded to try to make us sound better.}
Equation \ref{dunkel} is not well-known. 
While it follows easily from work of 
Otto Dunkel published in 1925,\footnote{To obtain (\ref{dunkel}) from Dunkel's work,
compare the equation for $P_2(n)$ in section 6 of \cite{dunkel1925solutions}
to the equation for 
$P_2(n)$ in section 10 of the same paper.} it 
was recently rediscovered; a combinatorial 
proof is given in  \cite{https://doi.org/10.48550/arxiv.2208.01224}.

The height of the tallest building, $H_{{\mathcal R},\,k}(d)$,  is $1$ fewer than the
number of leaves, completing the proof.
\end{proof}

\subsection{Bonus Eggs.}

For Standard Eggs or Replacement Eggs, starting with only $1$ egg is not interesting,
but for  Bonus Eggs starting with $1$ egg provides some useful information.
The tree in Figure~\ref{bonus.eggs} shows the number of eggs remaining
as we progress through the first four egg drops.
If you extend   the tree in Figure~\ref{bonus.eggs}  down to depth $9$, you will see
that the number of leaves with $0$ eggs 
%\todo[inline]{slight wording change: ``occurring" replaced by ``that occur"}
that occur  at depth $1,\ 3,\ 5,\ 7,\ \hbox{\rm and\ }9$, respectively,
is $1,\ 1,\ 2,\ 5,\ \hbox{\rm and\ }14$.  The first $5$ Catalan numbers
are  $1,\ 1,\ 2,\ 5,\ \hbox{\rm and\ }14$.  The Catalan numbers
appear because the $n$th Catalan number is the cardinality of the set of
sequences of $n\ \ {+}1$'s and $n\ \ {-}1$'s with non-negative partial sums. 
To run out of eggs at exactly depth $2n-1$ one must have broken $n$ eggs---those
are the $n\ \ {-}1$'s---and have received  $n-1$ bonus eggs---those are $n-1$ of the ${+}1$'s, the $n$th 
${+}1$ is the starting egg.
D. F. Bailey in \cite{Bailey_Counting} generalized the Catalan number construction
by considering   sequences of $n\ \ {+}1$'s and $m\ \ {-}1$'s with non-negative partial sums. 
We  further generalize Bailey's work\footnote{Relative to Bailey's paper, we will reverse the roles
of $m$ and $n$.} and apply the results in the present setting
of Bonus Eggs.  

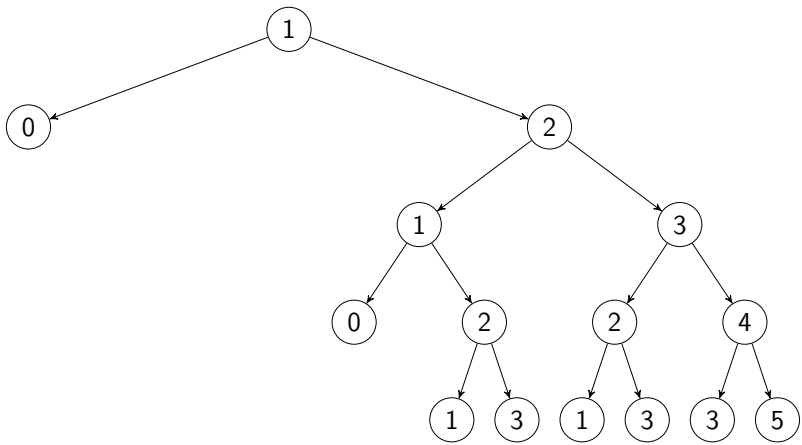
\begin{figure*}[htp]\centering    % TWO EGGS
%\resizebox{\textwidth}{!}{
\resizebox{4.25in}{!}{
\begin{tikzpicture}[->,>=stealth',level/.style={sibling distance = 16cm/2^#1, level distance = 1.5cm}] 

    \node [arn_rn] {1}
child{
   node [arn_rn] {0}
}
child{
   node [arn_rn] {2}
   child{
      node [arn_rn] {1}
      child{
         node [arn_rn] {0}
      }
      child{
         node [arn_rn] {2}
         child{
            node [arn_rn] {1}
         }
         child{
            node [arn_rn] {3}
         }
      }
   }
   child{
      node [arn_rn] {3}
      child{
         node [arn_rn] {2}
         child{
            node [arn_rn] {1}
         }
         child{
            node [arn_rn] {3}
         }
      }
      child{
         node [arn_rn] {4}
         child{
            node [arn_rn] {3}
         }
         child{
            node [arn_rn] {5}
         }
      }
   }
}

;
\end{tikzpicture}
}
\caption{Bonus Eggs.} \label{bonus.eggs}
\end{figure*}

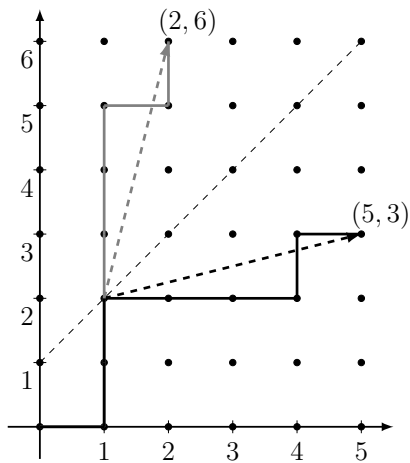
\begin{figure*}[htp]\centering    % TWO EGGS
\resizebox{2.25in}{!}{
\begin{tikzpicture}
    [%%%%%%%%%%%%%%%%%%%%%%%%%%%%%%
        dot/.style={circle,draw=black, fill,inner sep=1pt},
    ]%%%%%%%%%%%%%%%%%%%%%%%%%%%%%%

\newcommand\xmax{5}
\newcommand\ymax{6}

\tikzmath{
	\themin  = min(\xmax, \ymax);
};

% draw the grid
\draw[->,thick,-latex] (0,-.5) -- (0,\ymax+ .5);
\draw[->,thick,-latex] (-.5,0) -- (\xmax + .5,0);

\foreach \y in {0,...,\ymax}
	\foreach \x in {0,...,\xmax}
    	\node[dot] at (\x,\y){};
    
% draw the coordinates   
\foreach \x in {1,...,\xmax}
    \draw (\x,.1) -- node[below,yshift=-1mm] {\x} (\x,-.1);
\foreach \y in {1,...,\ymax}
    \draw (.1,\y) -- node[below,xshift=-2mm] {\y} (-.1,\y);

% centerline
%\draw (0,0) -- (\themin, \themin);

\draw[dashed] (0,1) -- (\themin, \themin + 1);

\foreach \z [remember=\xf as \xa  (initially 0), remember=\yf as \ya  (initially 0) ] in {1, -1, -1,   -1, -1, -1, 1, -1} {
	\tikzmath{
		\xs = \xa;
		\ys = \ya;
		\xf = \xa + (\z+1)/2;
		\yf = \ya  + (-\z+1)/2;
	}
	\draw[very thick,gray] (\xs,\ys) -- (\xf, \yf);
}

\foreach \z [remember=\xf as \xa  (initially 0), remember=\yf as \ya  (initially 0) ] in {1, -1, -1,    1, 1, 1, -1, 1} {
	\tikzmath{
		\xs = \xa;
		\ys = \ya;
		\xf = \xa + (\z+1)/2;
		\yf = \ya  + (-\z+1)/2;
	}
	\draw[very thick,black] (\xs,\ys) -- (\xf, \yf);
}

%\draw[very thick,blue] (2,3) -- (2, 6);
%\draw[very thick,blue] (2,6) -- (3, 6);
%\draw[very thick,blue] (3,6) -- (3, 7);

\draw[->,dashed,very thick] (1,2) --  (5, 3);
\draw[->,dashed,very thick, gray] (1,2) --  (2, 6);

\newcommand\len{.3}
\node[] at (5+\len,3+\len) {$(5,3)$};
\node[] at (2+\len,6+\len) {$(2,6)$};

\end{tikzpicture}

}
\medskip
\caption{Reflection of the solid path over $y=x+1$.} \label{reflect}
\end{figure*}

\begin{definition} Let $k$, $m$ and $n$ be 
non-negative integers with    $ n\leq m +  k$.
Let  $G_k(m,n)$ denote the number of sequences $a_1, a_2,\dots, a_{m+n}$ of 
$m$\hskip .45em$+1$'s and $n$\hskip.45em$-1$'s  for which every partial sum is 
greater than $-k$, that is,  % if  $1\leq i < m+ n$, then 
\begin{equation}\label{stay.alive}
\hbox{\rm if\ \ }1\leq i \leq m+ n, \hbox{\rm\ \ then\ \ }   a_1+a_2+\cdots+a_i > -k\,.
\end{equation}
Note  that because the hypothesis of (\ref{stay.alive}) is false when $m=n=0$,
the statement itself is satisfied by the empty sequence. Thus
$G_k(0,0) = 1$ holds for all non-negative $k$.
\end{definition} 

\begin{theorem}\label{bailey.variant.thm}
For $k$, $m$ and $n$  non-negative integers with  $n\leq m+k$,
it holds that
\begin{equation}\label{bailey.variant}
G_k(m,n) = 
\left\{\begin{array}{cl}
\displaystyle \binom{m+n}{n}  &\hbox{\rm if\ } n<k , \\[2ex]
\displaystyle \binom{m+n}{n} - \binom{m+n}{n-k}&\hbox{\rm otherwise}.
% \displaystyle \binom{m+n-1}{n-1} - \binom{m+n-1}{n-k-1}&\hbox{\rm if\ } 1\leq m  \hbox{\rm\ and\ }   n=m+k\,.
\end{array}
\right.
\end{equation}
\end{theorem}

\begin{remark}
Notice that   the theorem tells us that 
$G_1(m,m)$  equals the $m$th Catalan number, so 
it should not be surprising that 
our proof of  the  theorem is similar to  
a classic construction used in studying Catalan numbers. 
\end{remark}

\begin{proof}  
We will translate each partial sum of plus and minus $1$'s
into a path in the cartesian plane that begins at $(0,0)$ and, in some order,
takes $m$ steps to the right and $n$ steps up ending 
at $(m,n)$. We want to count the
number of such paths that do not touch the line $y=x+k$.

\medskip
\noindent
\textbf{Assuming \boldmath $n< k.$}

\noindent
The total number of paths that take $m$ steps to the right and $n$ steps up
is $\binom{n+m}{n}$. If  $n<k$, then no such path can reach the line
$y=x+k$. 

\medskip
\noindent
\textbf{Assuming} {\boldmath$\ k\leq n$.}

\noindent
As part of the definition  we have assumed  that $n\leq m+k$.
If $n= m+k$, then any path  that take $m$ steps to the right and $n$ steps up
ends on the line $y=x+k$. Thus $G_k(m,m+k) = 0$ and that agrees with the right-hand side
of (\ref{bailey.variant}). So from now on we assume that $n< m+k$.

\medskip
The total number of paths that take $m$ steps to the right and $n$ steps up
is $\binom{n+m}{n}$. Some of those paths may be ``bad'' paths that
touch the line $y=x+k$. So we need to count the bad paths,
%\todo[inline]{stuck the next clause onto the end of the sentence so as to announce the plan.}
and to do that we will show that the  bad paths
from $(0,0)$ to $(m,n)$
can be put into one-to-one correspondence with 
the  paths from $(0,0)$ to $(n-k,m+k)$.

Any bad path has a first point $(x,x+k)$ where it contacts the line $y=x+k$.
The vector from $(x,x+k)$ to $(m,n)$ is $(m-x,n-x-k)$. If 
in that vector we swap steps to
the right for steps up and swap steps up for steps to the right, we obtain the
vector $(n-x-k,m-x)$. Proceeding from $(x,x+k)$ via the new vector takes
us to $(n-k,m+k)$   (see Figure  \ref{reflect}), a point above the line $y=x+k$ because $n<m+k$.
On this new path from $(0,0)$ to $(n-k,m+k)$, the point $(x,x+k)$ is the
first point where the path hits the line. 

Every path of $n-k$ steps to the right and $m+k$ steps up will cross the
line $y=x+k$ somewhere, so there will be a first point of contact, say 
$(x,x+k)$. The vector from $(x,x+k)$  to  $(n-k,m+k)$ is 
$(n-k-x, m-x)$  Again swapping right steps and steps up, we obtain
the vector $(m-x, n-k-x)$.  Proceeding from $(x,x+k)$ via the new vector takes
us to $(m,n)$. On this new path from $(0,0)$ to $(m,n)$, the point $(x,x+k)$ is the
first point where the path hits the line.

From the above constructions, we see that  the number of bad paths 
$(0,0)$ to $(m,n)$
equals the number of paths with
 $n-k$ steps to the right and $m+k$ steps up, that is, $\binom{m+n}{n-k}$.
 It follows that
 $$
 G_k(m,n) = \binom{m+n}{n} - \binom{m+n}{n-k}
 \,.
 $$
\end{proof}

\begin{theorem}
With  Bonus Eggs starting with $k$ eggs, 
the tree of depth $d<k$ has $2^d$ leaves and $2^d-1$ inner nodes.
For the tree of depth $k\leq d$, set
$$
\alpha = \left\lfloor  \frac{d-k}{2}  \right\rfloor \hbox{\rm\ \ and\ \ } 
\beta =  \left\lfloor  \frac{d-k+1}{2}  \right\rfloor
$$ 
then the tree has
$$
Z_k(d) = \sum_{m=0}^\alpha  \frac{k}{2m+k}\ \binom{2m+k}{m}
$$
leaves with  no remaining eggs  and
$$
M_k(d) = \sum_{n= \beta}^{\beta+k-1} \binom{d}{n}
$$
leaves that still have eggs.

The height of the tallest building for which the strength of the 
eggs can be determined starting with $k$ eggs and  
using no more than $d$ egg drops
is  $2^d-1$ when $d<k$ and when $k\leq d$ is 
$$
H_{{\mathcal B},\,k}(d) =
\sum_{m=1}^\alpha   \frac{k}{2m+k}\ \binom{2m+k}{m}
+ 
\sum_{n= \beta}^{\beta + k -1} \binom{d}{n}
\,.
$$
\end{theorem}
\begin{proof}
The result is clear for $d<k$, so we will assume $k\leq d$.

\medskip
\noindent
\textbf{Paths in the tree that end with no eggs left.}
The first depth at which all $k$ eggs can be broken
is clearly $k$. Also notice that the number of eggs remaining
unbroken at depth $i$  always has the same parity as $i+k$. 
So we need to consider paths that end with no eggs left
at depth $k+2m$, where $m$ ranges from $0$ to $\lfloor (d-k)/2 \rfloor$.

\medskip
Each failure results in $-1$ egg
and each success results in $+1$ egg.
The supply of eggs is exhausted when there have been 
$k$ more failures than successes. Suppose the
supply of eggs was exhausted exactly on trial $k+ 2m$,
where $k+ 2m \leq d$.
Immediately before trial $k+2m$  we know there must have been exactly 
one egg left. Say on those previous
$k+2m-1$ trials there had been $a$ successes
and $b$ failures. To have one egg left, we must have $k+a-b = 1$. 
To account for $k+2m-1$ trials, we must have $a+b=k+2m-1$.
Solving those equations for $a$ and $b$, we find that 
$a=m$ and $b = m+k-1$.

Because we ran out of eggs on trial 
$k+2m$ and not  earlier, the $a$ successes
and $b$ failures can be thought of as a
sequence of 
$a$\hskip .45em$+1$'s and $b$\hskip.45em$-1$'s  
such that the partial sums  are all greater than $-k$.  
The number of such sequences is  
$G_k(a,b) = G_k(m,m+k-1)$.  

Thus the number of paths of depth not exceeding $d$ that
end with no eggs remaining equals
$$
\sum_{m = 0}^{\alpha} G_k(m, m+k-1)
\,.
$$

By (\ref{bailey.variant}) we see that for $1\leq m$
\begin{eqnarray*}
G_k(m,m+k-1) &=& \binom{2m+k-1}{m+k-1} -\binom{2m+k-1}{m-1}\\[0.5ex]
&=&  \frac{(2m+k-1)!}{m!\ (m+k-1)!} -  \frac{(2m+k-1)!}{(m-1)!\ (m+k)!} \\[0.5ex]
&=&  \frac{(2m+k-1)!}{m!\ (m+k)!}\ \Big[  (m+k)  -  m\Big]\\[0.5ex]
&=& \frac{k}{2m+k}\ \binom{2m+k}{m}
\,.
\end{eqnarray*}
Also note that the equation $G_k(m,m+k-1) = \frac{k}{2m+k}\ \binom{2m+k}{m}$ remains valid when $m=0$.

\medskip
\noindent
\textbf{Paths in the tree  that reach depth \boldmath{$d$} with eggs still remaining.}
For this part of the proof, let $n$ be the number of failures that occur in the $d$ experiments.
The number of successes is $d-n$, and because  there are still eggs remaining unbroken
at depth $d$, we have $1\leq k+  (d-n)  -n$. Thus the values of $n$ range from $0$
up to  $\lfloor (d+k-1)/2 \rfloor $. Let $n^*$ denote this last value.

A path in the tree that reaches depth $d$ with eggs remaining corresponds to 
sequence of 
$d-n$\hskip .45em$+1$'s and $n$\hskip.45em$-1$'s  
such that the partial sums  are all greater than $-k$.  
The number of such sequences is  
$G_k(d-n,n)$.  

We see that the number of nodes at depth $d$ where there
are still remaining eggs is
\begin{eqnarray*}
M_k(d) &=& \sum_{n=0}^{n^*} G_k(d- n,n)  %\\
%&=&
=
\sum_{n=0}^{k-1} \binom{d}{n}  
+ \sum_{n=k}^{n^*}  \left[ \binom{d}{n}
-  \binom{d}{n-k}\right]\\  
&=&
\sum_{n=0}^{n^*} \binom{d}{n}  
- \sum_{n=k}^{n^*}  \binom{d}{n-k}  %\\  
%&=&
=
\sum_{n=0}^{n^*} \binom{d}{n}  
- \sum_{n=0}^{n^*-k}  \binom{d}{n} \\  
&=&
\sum_{n = n^*-k+ 1}^{n^*} \binom{d}{n}  
\,.
\end{eqnarray*}
It is then immediate that $\lfloor (d+k-1)/2 \rfloor  - k+1 
= \lfloor (d-k+1)/2 \rfloor =\beta$. 
\end{proof}

%\todo[inline]{The references to Pilling, Alves, \& Innes and to our sums paper need fleshing out.}

\bibliographystyle{alphaurl}
\bibliography{eggs,fibquart}
\end{document}